\documentclass[12pt]{article}
\usepackage{a4wide}
\usepackage{epsfig}
\usepackage{amsmath,amsfonts,amssymb,amsthm}
\usepackage{eucal}
\usepackage{graphics,graphicx}
\usepackage{float}
\usepackage{xcolor}
\usepackage{authblk}
\usepackage{subfig}
\usepackage{epstopdf}
\usepackage{authblk}
\usepackage{hyperref}
\newtheorem{theorem}{Theorem}[section]

\newtheorem{cor}{Corollary}[section]
\newtheorem{definition}{Definition}[section]
\newtheorem{lemma}{Lemma}[section]
\newtheorem{rem}{Remark}[section]
\newtheorem{ex}{Example}[section]
\begin{document}	
\title{Herman Rings: Structure, Dynamics, and Open Problems}
\author[]{Gorachand Chakraborty\footnote{gorachand.chakraborty@skbu.ac.in, Department of Mathematics, Sidho-Kanho-Birsha University, Purulia, India.} }
\author[]{Subhasis Ghora\footnote{subhasisghora06@gmail.com (Corresponding author), Department of Sciences and Humanities, Indian Institute of Information Technology Design and Manufacturing, Kancheepuram, Chennai, India. }}
\author[]{ Tarakanta Nayak\footnote{tnayak@iitbbs.ac.in, School of Basic Sciences, Indian Institute of Technology Bhubaneswar, Bhubaneswar, India. }}


\affil{}
\date{}
\maketitle	
	
\begin{abstract}
	
A detailed and understandable summary of the core discoveries and recent developments on the Herman ring of rational and transcendental meromorphic functions is presented.  Herman rings stand out from other periodic Fatou components and significantly determines the overall dynamics of a function. This is demonstrated in this survey. The results are discussed along with some research problems.
\end{abstract}

Mathematics Subject Classification(2010) 37F10,  37F45
\tableofcontents	
\section{Introduction}
A rational function is an analytic self-map of the Riemann sphere $\widehat{\mathbb{C}}$. A transcendental meromorphic function is an analytic function with possible poles and a single essential singularity, which we choose to be at $\infty$. Let $f$ be a rational with degree at least two or a  transcendental meromorphic function. We refer to such a function as meromorphic throughout this article, if not specified otherwise.  The $n$-times composition of $f$ with itself is denoted by $f^n$. The study of $\{f^n\}_{n>0}$, i.e., the iterative behavior of $f$ is the concern of \textit{Complex dynamics}, also known as \textit{Holomorphic dynamics}. 

The Fatou set of $f$ is the set of all points in a neighborhood of which $\{f^n\}_{n\geq0}$ is normal in the sense of Montel.
It is an open set by definition. A maximally connected subset of the Fatou set is called a Fatou component. A Fatou component $U$ is periodic if $f^p (U) \subseteq U$ for some $p$. The smallest such $p$ is called the period of $U$. A   $1$-periodic Fatou component is called invariant. Analogously, a point  $z_0 \in \widehat{\mathbb{C}}$ is called a $p$-periodic point of $f$ is $p$ is the smallest number for which $f^p (z_0)=z_0$. There are important relationships between certain type of periodic Fatou components and periodic points.  A periodic Fatou component $U$ is called an attracting domain or a parabolic domain if $f^{np} \to a$ on $U$ where $a$ is an attracting  or a rationally indifferent $p$-periodic point of $f$, i.e., $|(f^p)' (a)|<1$ or $=1$  respectively. The number $(f^p)' (a)$ is called the multiplier of $a$.  It is possible for a $p$-periodic Fatou component $U$ that $f^{np} \to a$ on $U$ where $f^p (a)$ is not defined. This can happen only when $f$ is transcendental meromorphic and $f^j(a)=\infty$ for some $0<j<p$. Such Fatou components are known as Baker domains. Note that the limit point  of $\{f^{np}\}_{n>0}$ on $U$ is  a constant whenever $U$ is an attracting, parabolic or Baker domain.  There is also another interesting possibility for $U$, namely rotation domain  where $f^{np}$ does not converge to any constant on $U$. More precisely, a $p$-periodic Fatou component $U$ is called a rotation domain if there exists a conformal map $\phi : U \rightarrow  X$ where $X=\{z: |z|<1\}$ or $\{z:1<|z|<r\}$ for some $r <\infty$ such that  $\phi(f^p(\phi^{-1}(z)))=e^{i 2 \pi \alpha}z$ for some irrational number $\alpha$.
The Fatou component $U$ is called a Siegel disc or a Herman ring if $X$ is the disc $\{z: |z|<1\}$ or the annulus $\{z:1<|z|<r\}$  respectively. 

\par  Herman rings stand out from other periodic Fatou components in a number of ways. The most obvious one arises from the fact that $f^p$ on a $p$-periodic Herman ring $H$ is conformally conjugate to an irrational rotation of an annulus $A_r =\{z: 1< |z|<r\}$, i.e., for a conformal map $\phi: H \to A_r$, we have $\phi(f^p (\phi^{-1}(z)))= e^{i 2 \pi \alpha}z$ for some irrational number $\alpha$.  As each circle $S= \{z: |z|=r'\}$, for $1<r'<r$ is invariant under the rotation $z \mapsto e^{i 2 \pi \alpha}z$ and $\phi^{-1}(S)$ is an  $f^{p}$-invariant Jordan curve in $H$, the ring $H$  is an uncountable union of $f^{p}$-invariant Jordan curves. Although a $p$-periodic Siegel disk is likewise an uncountable union of $f^{p}$-invariant Jordan curves, a Herman ring differs in that each such curve separates the two components of $\widehat{\mathbb{C}}\setminus H$ in such a way that both components contain points of the Julia set.
In fact, Herman rings are  doubly connected whereas Siegel discs are simply connected. Some   distinctive aspects of Herman rings follow.
\begin{itemize}
	\item (Connectivity) The Julia set of a meromoprhic function is the complement of its Fatou set in $\widehat{\mathbb{C}}$.  It is well-known that an attracting or parabolic domain is either simply or infinitely connected. But a Herman ring 
is always doubly connected. Therefore, in the presence of Herman ring, the Julia set contains at least two components (i.e., maximally connected subsets of the Julia set), namely those containing the two boundary components of the Herman ring. In particular, the Julia set is disconnected in the presence of a Herman ring. In fact, it follows that there are uncountably many Julia components (see Corollary 4.5, \cite{jm}). 
\item (Buried Julia components) A Julia component is called buried if it is not the boundary of any Fatou component. Since each invariant Jordan curve in a Herman ring separates the sphere into two regions containing Julia points, Herman rings inherently create nested Julia structures. This separation mechanism suggests that Herman rings may contribute to the formation of Julia components that are disjoint from the boundary of Fatou compotents. 

\item (Unrelated to periodic points) A $p$-periodic attracting or parabolic domain contains an attracting or rationally indifferent $p$-periodic point in its interior or on its boundary respectively. Similarly, a $p$-periodic Siegel disk always contains a $p$-periodic point, which is the pre-image of $0$ under the conjugating map $\phi$. But a Herman ring can not contain any periodic point.  Constructing functions with an attracting or rationally indifferent  periodic point  has been a standard way to produce functions with attracting or parabolic domains. Similarly, an irrationally indifferent periodic point   with multiplier $e^{i 2 \pi \alpha}$  where $\alpha$ is of Brjuno type  leads to a Siegel disk. This does not work for constructing Herman rings  and  alternative tools and methods are required.
	\item (Injectivity) Like Siegel discs, the function is always injective  on its Herman ring. An essential consequence is that there are infinitely many Fatou components whenever there is a Herman ring. This rules out the possibility of a totally disconnected Julia set.
	\item (Completely invariant Fatou components) A Fatou component $U$ of $f$ is completely invariant if it is invariant  and $f^{-1}(U) \subseteq U $. A completely invariant Fatou component cannot be a Herman ring. It is known that the boundary of a completely invariant Fatou component is the Julia set and all other Fatou components (if any) are simply connected. Thus, Herman ring cannot coexist with a completely invariant Fatou component.
\end{itemize}

 This article surveys various results on Herman rings. Section \ref{diff} describes the construction of Herman rings arising from circle diffeomorphisms. Section \ref{quasi} is devoted to Shishikura’s construction of Herman rings via quasiconformal surgery. Sections \ref{rational} and \ref{trans} discuss the existence of Herman rings for rational maps and transcendental functions, respectively. Finally, Section \ref{con} concludes with a discussion of several open problems. Throughout this article, the term ring will refer exclusively to a Herman ring.
 
 Our primary aim here is to describe the present state of research related to Herman rings for both rational and transcendental meromorphic functions. The article is mostly an exposition of the research done in the field. A few illustrations have been provided to give visual descriptions of some of the results. Some unsolved questions and conjectures for future research have also been presented.

\section{From circle diffeomorphisms  to Herman rings}\label{diff}
 The existence of all types of  periodic Fatou components except Herman
rings were known before 1950 (see \cite{cayley1879application,fatou1926iteration, leau1897etude,siegel1942iteration}). Herman rings were initially conjectured to not exist for
rational functions by Fatou (see \cite{pdnf}). However, it was in 1979, when M.R. Herman gave an
example of a function having Herman ring  \cite{herman1984exemples}.  Herman's idea was based on extension of circle diffeomorphisms.
This section describes the route map to extend an analytic linearizable circle map to a Herman ring.

Let $S^1$ be the unit circle with the counterclockwise orientation.  A homeomorphism $f: S^1\to S^1$ is called  orientation preserving if  $(f(x),f(y),f(z))$
are  in counterclockwise order whenever  $(x,y,z)$ are in counterclockwise order on $S^1$ for distinct $x,y,z$.  A standard way to understand circle homeomorphisms is by lifting them to the real line.  A map $F:\mathbb{R}\to\mathbb{R}$ is called a lift of $f$ if $\Pi\circ F=f\circ\Pi,$ where $\Pi:\mathbb{R}\mapsto S^{1}$ is given by $\Pi(x)=e^{2\pi ix}$.

\begin{rem}
By definition, a lift satisfies the following diagram.  
		\begin{figure}[H]
			\centering
			\includegraphics[width=0.2\textwidth]{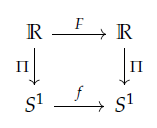}  
			\caption{The lift of the unit circle.} 
			\label{dia}
		\end{figure}

	\end{rem}	
For every circle homeomorphism $f$, there always exist a lift unique up to the addition of an integer \cite{may}. It also follows that $F(x+1)=F(x)\pm 1$ for every lift of $f$. A simple example of a circle homeomorphism is the rigid rotation $R_{\theta}(z)= e^{2 \pi i \theta}z$ by an angle $\theta$. We are interested in homeomorphisms that are conjugate to rigid rotations.  A fundamental topological invariant of circle homeomorphisms is the rotation number. 
\begin{definition}\textbf{(Rotation number)} Let $F$ be a lift of an orientation-preserving homeomorphism $f:S^1\to S^1.$  The rotation number $\rho(f)$ of $f$ is defined as $ \lim_{n\to\infty}\frac{F^n(x)-x}{n} $ where the limit is taken $\text{mod~} 1$.
	
\end{definition}
Note that the limit in the above definition actually exists. Further, $\rho(f)$   is independent of $x$ and gives an idea of the average rotation of a point under $f$. Sometimes it is useful to
work with an equivalent definition:
$\rho(f)=(\lim_{n\to\infty}\frac{F^n(x) }{n})~(\text{mod~} 1)$.  It is well-known that for an orientation-preserving
homeomorphism   $f:S^1\to S^1$, the rotation number $\rho(f)$ is irrational if and only if $f$ has no periodic point  (see for example Chapter 7, \cite{brin-stuck}). We are  concerned with circle diffeomorphisms with irrational rotation number.  

Considering  $F(x)=x+\theta$ as a lift of the rigid rotation $R_{\theta}$ of the unit circle $S^1$ for $\theta \in (0, 2 \pi]$, it is seen that the rotation number  $\rho(R_\theta)$ is indeed $\theta$.
 A set of interesting examples different from rigid rotations is given by the Arnold family  $x \mapsto x+\alpha +\beta \sin(x)$ where the rotation number is a function of both $\alpha$ and $\beta$.
 
An orientation preserving circle diffeomorphism $f$ can be conjugate to a rigid rotation. The nature of the conjugating map significantly depends on the level of smoothness of $f$ as well as on the nature of the rotation number. This subject has been studied by Denjoy and Arnold culminating - from the perspective of this article -  with the work of Herman, who proved the following in 1979 (see \cite{yoc} for a discussion).  An irrational number $x$ is called a Diophantine number of order $k\geq 2$ if there exists $\epsilon >0$ such that $|x-\frac{p}{q}|>\frac{\epsilon}{q^k}$ for all rational numbers $\frac{p}{q}$.

\begin{theorem}[\cite{herman1984exemples} ]
	Let \( f \) be a 	\( C^\infty \)  orientation-preserving diffeomorphism of the circle  with irrational rotation number \( \rho(f) \). If \( \rho \) is Diophantine then \( f \) is conjugate to \( R_\rho \) where the conjugating map is \( C^\infty \).
\end{theorem}

%
%
%
  An orientation preserving homeomorphism $f:S^1\to S^1$ is said to be analytically linearizable if there is an analytic map $\phi: S^1\to S^1$ such that $\phi \circ f \circ \phi^{-1}$ is a rigid rotation.  The theorem above gives that every analytic circle diffeomorphism with Diophantine rotation number is linearizable.
 There is a superset of Diophantine numbers, referred to as Herman numbers by  Yoccoz, who proved that an analytic diffeomorphism $f$ of the circle   is analytically linearizable if  the rotation number of \(f\) is of Herman type. 
Such $f$ can be shown to have an invariant annulus containing the unit circle on which it is analytically conjugate with a rigid rotation. This invariant annulus is in fact an invariant Herman ring for $f$. For a concrete example, see Section 6.1.1 ~\cite{branner2014quasiconformal}.  

\section{Shishikura's construction of Herman rings}\label{quasi}

A continuous complex-valued function \( f \) defined on an interval \( I \subset \mathbb{R} \) is said to be absolutely continuous on \( I \) if it satisfies the following: For every \( \varepsilon > 0 \), there exists \( \delta > 0 \) such that $\sum_{j=1}^{n} \left| f(b_j) - f(a_j) \right| < \varepsilon$ for every finite collection $\{(a_j, b_j): 1 \leq j \leq n\}$ of pairwise disjoint intervals \( (a_j, b_j) \) whose closures are contained in \( I \) such that $\sum_{j=1}^{n} |b_j - a_j| < \delta$. For  two given domains $U$ and $V$ in $\mathbb{C}$ and a real number $K\geq1$, a homeomorphism $\phi: U\rightarrow V$ is called a $K$-quasiconformal mapping (qc-mapping) if $\phi $ is absolutely continuous on almost all lines parallel to the coordinate axes, and if $|\mu_\phi|=|\frac{\partial_{\bar{z}} \phi}{\partial_{z} \phi}|\leq \frac{K-1}{K+1}$ almost everywhere with respect to the Lebesgue measure where 
$\partial_{\bar{z}} \phi = \frac{1}{2} \left( \frac{\partial \phi}{\partial x} + i \frac{\partial \phi}{\partial y} \right)  \mbox{~and~}\quad
\partial_{z} \phi = \frac{1}{2} \left( \frac{\partial \phi}{\partial x} - i \frac{\partial \phi}{\partial y} \right)$. Here $\mu_\phi$ is called the Beltrami coefficient of $\phi$. 
\par A mapping \( g : U \to \mathbb{C} \) is \( K \)-quasiregular if  \( g \) can be expressed as $g = f \circ \varphi$, where \( \varphi : U \to \varphi(U) \) is \( K \)-quasiconformal and \( f : \varphi(U) \to g(U) \) is holomorphic. A quasiregular map \( f: \widehat{\mathbb{C}} \to \widehat{\mathbb{C}} \) is called quasirational if it is quasiconformally conjugate to a rational map.

Quasiconformal surgery is a powerful technique for constructing new rational functions from existing ones by carefully modifying them while preserving essential aspects of their dynamics. A central obstacle in this process is the Identity Theorem, which prevents us from directly gluing together different analytic functions. Since a holomorphic function is entirely determined by its values on any open set, naive gluing would fail to produce a well-defined analytic map. To overcome this, we temporarily relax the requirement of analyticity and instead work with quasiconformal maps, to define conjugacies between parts of the domain. We can create a globally defined quasiregular map by gluing together pieces using these quasiconformal conjugacies. The key question is: Can this quasiregular map be conjugated back to a rational function? The answer, under the right conditions, is yes. This is made possible by the Measurable Riemann Mapping Theorem \cite{ahlfors2006lectures}, which guarantees that if we prescribe a Beltrami coefficient (a measure of distortion) that is invariant under the dynamics, then there exists a quasiconformal homeomorphism conjugating our map to a genuine rational function. 

\begin{theorem}\textbf{(Measurable Riemann Mapping Theorem)}
Let \( U \subset \mathbb{C} \) be an open set such that \( U \cong \mathbb{D} \) (respectively \( U = \mathbb{C} \)). Let \( \mu : \mathbb{C} \to \mathbb{C} \) be a measurable function satisfying $\|\mu\|_{\infty} = k < 1$. Then \( \mu \) is integrable; that is, there exists a quasiconformal homeomorphism \( \varphi : U \to \mathbb{D} \) (respectively onto \( \mathbb{C} \)) which solves the Beltrami equation, i.e.,
\[
\mu(z) = \frac{\partial_{\bar{z}} \varphi(z)}{\partial_z \varphi(z)}
\quad \text{for almost every } z \in U.
\]
Moreover, \( \varphi \) is unique up to post-composition with automorphisms of \( \mathbb{D} \) (respectively \( \mathbb{C} \)).

\end{theorem}
Another tool used in  Shishikura's construction is the Fundamental Lemma for Quasiconformal Surgery \cite{branner2014quasiconformal,smm}. 
\begin{theorem}\textbf{(Fundamental Lemma)}\label{fl}
    Let \( f : \widehat{\mathbb{C}} \to \widehat{\mathbb{C}} \) be quasiregular.  
For \( p \geq 1 \), suppose there exist:

\begin{enumerate}
    \item \( U = U_1 \cup \cdots \cup U_p \), consisting of \( p \) disjoint open subsets \( U_j \subset \hat{\mathbb{C}} \) such that, for \( 1 \leq j < p \),   $f(U_j) = U_{j+1}, \quad f(U_p) \subset U_1$;

    \item A quasiconformal homeomorphism \( \psi : U \to \tilde{U} \), where \( \tilde{U} \subset \hat{\mathbb{C}} \) (the gluing map);
    
    \item A quasiregular map \( H : \tilde{U} \to \tilde{U} \) such that \( H^p \) is holomorphic;
\end{enumerate}
such that \( f|_U = \psi^{-1} \circ H \circ \psi \) and \( \partial_{\bar{z}} f = 0 \) almost everywhere in \( \hat{\mathbb{C}} \setminus f^{-N}(U) \), for some \( N \geq 0 \). Then \( f \) is quasirational.
\end{theorem}

In essence, quasiconformal surgery allows us to construct rational functions with prescribed dynamical features (such as a Herman ring) by
\begin{enumerate}
    \item Starting with a model dynamical system (often simpler or locally known behavior).
    \item Gluing together different parts using quasiconformal maps.
    \item Applying quasiconformal theory to obtain a globally defined rational map with some specific behavior.
\end{enumerate}
We describe the outlines of the construction. The intention is to create a rational map $f$ with a $p$-periodic Herman ring for a given natural number $p$. First, consider a rational map $f_0$ with a $p$-periodic Siegel disk $S$. Let $\{S=S_1, S_2,....,S_p\}$ be the cycle of $S$. Let $\gamma_1$ be an $f_0^p-$ invariant Jordan curve in $S_1$ and $\gamma_{i+1}=f_0(\gamma_{i})$ for $i=1,2,...,(p-1)$. Now consider $p$ many rational maps $f_1, f_2,...,f_p$ such that $f_p\circ f_{p-1}\circ ...\circ f_1$ has an invariant Siegel disk $S_1'$ with rotation number $-\theta$. Let $\gamma_1'$ be a $f_p\circ f_{p-1}\circ ...\circ f_1$ invariant Jordan curve in $S_1'$ (Figure \ref{3periodic} describes the case for $p=3$). Define $S_{i+1}'=f_i(S_i')$ and $\gamma_{i+1}'=f_i(\gamma_{i}')$ for $i=1,2,...,(p-1)$.  It can be shown that there exist quasiconformal mappings $\phi_1, \phi_2,...,\phi_p: \widehat{\mathbb{C}} \rightarrow \widehat{\mathbb{C}}$ such that $\phi_i(\gamma_{i})=\gamma_{i}',~~\phi_{i+1}\circ f_0=f_i\circ \phi_i$ on $\gamma_{i}$ and $\phi_i$ is conformal in the neighborhood of $\widehat{\mathbb{C}}\setminus (S_i\cap \phi^{-1}(S_i'))$ for $i=1,2,...,p$ where $\phi_{p+1}=\phi_1$. Define a function \( g :\widehat{\mathbb{C}} \to \widehat{\mathbb{C}}  \) by
\[
g =
\begin{cases}
f_0 & \text{on } \widehat{\mathbb{C}}  \setminus  \cup_{i}B(\gamma_{i}), \\
\psi_{i+1}^{-1} \circ f_i \circ \psi_i & \text{on } B(\gamma_{i}) .
\end{cases}
\]

Note that $B(\gamma)$ is the bounded complementary component of $\gamma$. It can be proved that \( g \) is continuous. Moreover, some suitable domain can be found such that the function $g$ will satisfy all the properties of Theorem \ref{fl} and hence, the function $g$ will be quasi-regular.
%
%
%
%
%
%
Hence, there exists a quasiconformal mapping $\phi$ such that $f=\phi\circ g\circ \phi^{-1}$ is a rational map with a $p$-periodic Herman ring of rotation number $\theta$. Specific example of such construction can be found in \cite{branner2014quasiconformal,pdnf}. Although quasiconformal surgery was originally used only for rational functions, it has been extended to transcendental meromorphic functions\cite{pdnf}.
\begin{figure}[h]
\centering
{\includegraphics[width=14.5cm,height=6cm,angle=0]{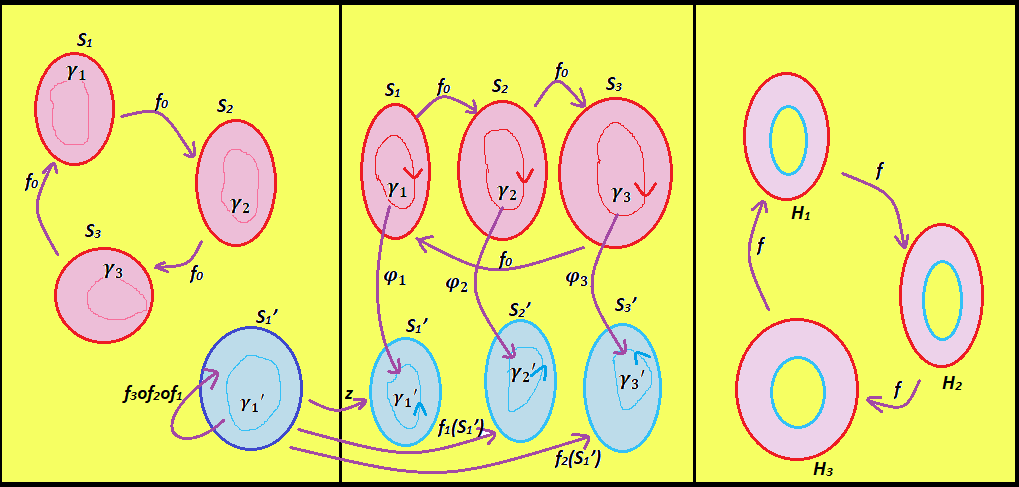}}
\caption{Construction of a $3$-periodic Herman ring.}
\label{3periodic}
\end{figure}

In \cite{FG2003}, Fagella and Geyer considered the standard (or Arnold) family of circle maps $F_{\alpha,\beta}(x) = x + \alpha + \beta \sin(x) \pmod{2\pi}$, for $ x, \alpha \in [0, 2\pi)$, and $ \ \beta \in (0,1)$. Using the quasiconformal surgery, they proved that its complexification $f_{\alpha,\beta}(z) = z e^{i\alpha} \exp\Big[\frac{1}{2} \beta \Big(z - \frac{1}{z}\Big)\Big]$ possesses a Herman ring for suitable values of $\alpha$ and $\beta$. Following the similar approach, Dom{\'\i}nguez and Fagella \cite{pdnf} showed that there exist $a, b \in \mathbb{C} \setminus \{0\}$ such that the meromorphic map $f(z) = \frac{a z^2 e^{-b z}}{z + 1}$ has a Herman ring. Their construction began with the map $g(z) = e^{2\pi i \theta} z e^{-z}$, where $\theta \in \mathbb{R} \setminus \mathbb{Q}$. Similarly, Yang \cite{yang} constructed a function
$f_{a,b}(z) = u \cdot \frac{z - b}{z - a} \cdot z^2 e^z$ having a $p$-cycle of Herman rings for some $a, b \in \mathbb{C}$ and $u = u(a,b) \in \mathbb{C} \setminus \{0\}$, starting from the map $g(z) = \lambda z^2 e^z$. It can be verified that none of these resulting maps omit any value. We believe that this is so as the starting maps themselves do not have omitted values and the fixed point  is itself an exceptional point of the function. We conjecture that to produce a map with an omitted value via quasiconformal surgery, one must start with a function that already has an omitted value and a $p$-cycle of Siegel discs. In that case, the resulting map would have a $p$-cycle of Herman rings and retain an omitted value.

\section{Herman rings for rational maps}\label{rational}
We start this section by proving that a certain class of functions does not have any Herman rings.
\begin{theorem}
\label{nohr}
    If $f$ is a polynomial (or an entire function), then $f$ has no Herman rings.
\end{theorem}
\begin{proof}
    Let $f$ be a polynomial (or an entire function), and if possible, let $H$ be a $p$-periodic Herman ring of $f$. Also, let $\gamma$ be an $f^p$-invariant Jordan curve of $H$. Since $f^p(\gamma)=\gamma$, either $f^p(B(\gamma))=B(\gamma)$ or $f^p(B(\gamma))=\widehat{\mathbb{C}}\setminus (B(\gamma)\cup \gamma)$. The map  $f^p$ is a polynomial (or an entire function); hence, $B(\gamma)$ does not contain any pole of $f^p$. Thus $f^p(B(\gamma))=B(\gamma)$ giving that $f^{np}(B(\gamma))=B(\gamma)$ for $n\in \mathbb{N}$. This is not possible as $B(\gamma)\cap J(f)\neq \emptyset$ giving that $\cup_{n\in \mathbb{N}}f^{np}(B(\gamma))$ is dense in $\widehat{\mathbb{C}}$. This concludes that $f$ does not have any Herman ring. 
\end{proof}

Though polynomials do not have Herman rings, rational functions with poles can have Herman ring. Lyubich has studied the problem of the existence of Herman rings of a rational map \cite{lm} in $1986$, who considers the family $\{z\mapsto 1+\frac{1}{wz^2}:w\in \mathbb{C}\setminus \{0\} \}$. He asked whether the maps in this family have Herman rings. Shishikura \cite{smm} in $1987$, using quasiconformal surgery and careful analysis of the behavior of critical points and fixed points, proved that the number of Herman rings can be at most $d-2$ for a rational function of degree $d$. This gives that the  quadratic rational maps have no Herman rings, which answers Lyubich's question as a particular case. Therefore, if a rational map has a Herman ring, its degree is at least three. Figure \ref{fighr} shows the Herman ring of the cubic rational function $f(z)=e^{2\pi i \frac{\sqrt{5}-1}{2}}\frac{z^2(z-4)}{1-4z}$.

\begin{figure}[H]
    \centering
    \includegraphics[width=13.5cm,height=8.5cm,angle=0]{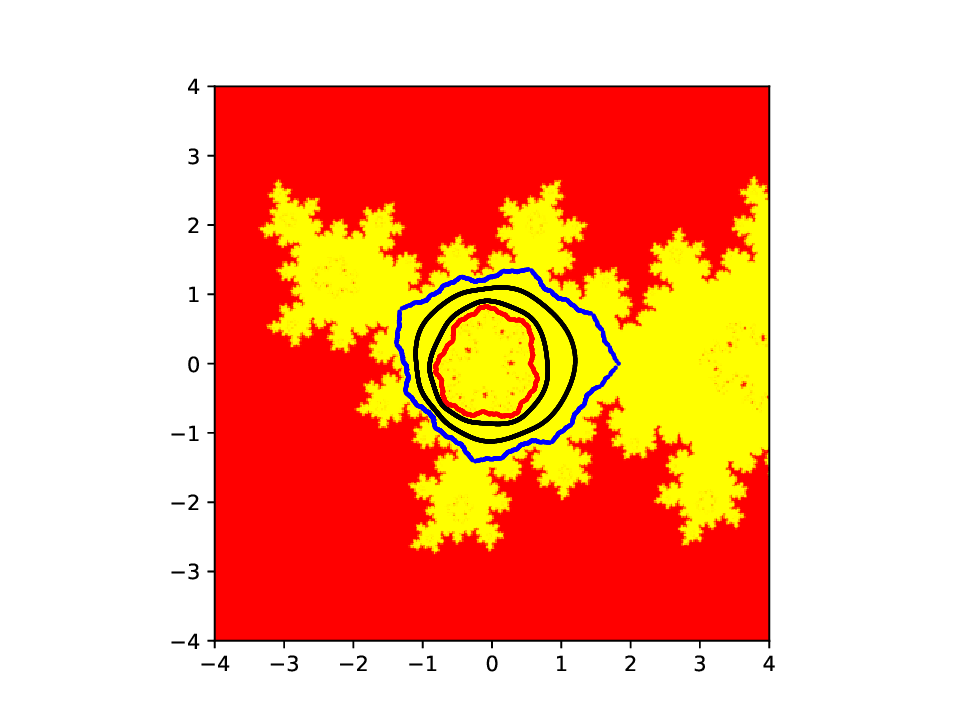}
    \caption{Herman ring of the function $e^{2\pi i\frac{\sqrt{5}-1}{2}}\frac{z^2(z-4)}{1-4z}$. The blue curve is the orbit of the critical point $\frac{19+\sqrt{105}}{16}$ and the red curve is the orbit of the critical point $\frac{19-\sqrt{105}}{16}$. The Herman ring is the yellow region lying between these two curves. The black curves are some invariant Jordan curves in the Herman ring.}
    \label{fighr}
\end{figure}
\par
\subsection{Singular values and Herman rings}

The singular values of a function $f$ are defined as the closure of the union of its critical values, that is, the images of points where $f'(z)=0$, and its asymptotic values, which are finite limits of $f(z)$ along curves tending to infinity (see \cite{berg93} for precise definition). The forward orbits of these singular values form the postsingular set, whose closure plays a central role in determining the global dynamics. Note that rational functions do not have any asymptotic value. A fundamental theorem of Fatou--Julia theory asserts that every periodic Fatou component is associated with the orbit of at least one singular value. Attracting and parabolic Fatou components necessarily contain singular values in their basins, and their stability arises from the local contraction induced by these singular orbits. 
For transcendental entire functions, additional components such as Baker domains, where all orbits tend to infinity, and wandering domains, which never repeat under iteration, may occur; although these components may not contain singular values, their existence and evolution are still constrained by the accumulation behavior of the postsingular set. 
In contrast, rotation domains-Siegel discs and Herman rings- are distinguished by the complete absence of singular values in their interiors. 
Singular values instead accumulate on the boundary of a rotation domain, where they prevent global linearization and determine the arithmetic conditions under which such domains exist. 
This boundary-driven interaction makes rotation domains (whose existence is highly sensitive to the distribution of the orbits of singular values) fundamentally different from all other Fatou components.

The dynamics of a rational map is dominated by its critical orbits. The boundary of the Herman ring is accumulated by at least one of the critical orbits in the Julia set. First, we define the $\omega$-limit set of a point. For  $c\in \mathbb{C}$ and a rational function $f$, the $\omega$-limit set of $c,$ denoted by $\omega(c)$, is defined by
$\omega(c)=\{z\in\widehat{\mathbb{C}}:$ there exists a subsequence $\{n_k\}$ such that $f^{n_k}(c)\to z\ as\ k\to \infty\}$.  The point $c$ is called a recurrent point of $f$ if $c\in\omega(c)$. The points that are not recurrent are called non-recurrent points. If a critical point is recurrent, then it is called a recurrent critical point. The image of a recurrent critical point is referred to as a recurrent critical value. Note that a recurrent point $a$ is in the Fatou set if and only if $a$ is in the Herman ring or the  Siegel disk. Consequently, if $a$ is a recurrent critical point of $f$, then $a$ is in the Julia set of $f$ as rotational domains do not contain any critical point. 
\begin{lemma}
For a rational function $f$, if $\Gamma$ is the boundary of a Siegel disk or a connected component of the boundary of a Herman ring, then there exists a recurrent critical point $c$ such that $\Gamma\subset \omega(c)$ \cite{mane}.	
\end{lemma}
 One can believe that a rational map has no Herman rings provided it has relatively simple critical orbits. Recall that the forward orbit of a point $z$ under the map $f$ is defined as $\{f^{n}(z): n\in \mathbb{N}\cup \{0\}\}$. It is infinite if the cardinal number of the set is not finite. Yang showed the following result \cite{yang2017rational} in $2017$.
\begin{theorem}\label{noherman}
	Let $f$ be a rational map with degree at least two. If there exists at most one infinite critical orbit in the Julia set of $f$ essentially, then $f$ has no Herman rings.
\end{theorem} 

Here, \textquotedblleft essentially" means if $c_1$ and  $c_2$ are two critical points with infinite orbits, then the forward orbit of $c_1$ and that of $c_2$ are disjoint. As an immediate corollary of the above theorem, we have the rational map $f_w(z)=1+\frac{w}{z^d}$ has no Herman rings for $d\geq 2$ and $w\in \mathbb{C}\setminus\{0\}$. This is so as the set of critical points of $f_w(z)=1+\frac{w}{z^d}$ is $\{0, \infty\}$. But $f_w$ has essentially only one critical orbit $0\mapsto \infty \mapsto 1\mapsto 1+w\mapsto \cdots $. Hence, by Theorem \ref{noherman}, the function $f_w$ has no Herman rings. This was proved by  Bam\'{o}n and Bobenrieth  \cite{bb} in $1999$. 
In $2000$, Milnor established a result more general than that of \cite{bb}, proving that rational maps with exactly two critical points do not possess Herman rings \cite{Milnor2000}.
Xiao et al. \cite{ywq} in $2010$ considered the family of functions $f_{m,d}(z)= z^m+\frac{\lambda}{z^d}$, where $\lambda\in \mathbb{C}\setminus \{0\} \}$ and $m$, $d$ are integers such that $m, d \geq 2$. It can be seen than the number of critical points is more than $2$, in fact, it has $m+d$ many simple critical points along with $0$ and $\infty$. So, Theorem \ref{noherman} can not be used here. Using careful analysis of invariant Jordan curves and winding number of such curves, the following is proved.
\begin{lemma}
\label{morecri}
	The function $f_{m,d}$ has no Herman rings.
\end{lemma}
 Consider the one-parameter family of functions,   $S_\lambda(z)=\biggl(\frac{(z+\lambda-1)^d+(\lambda-1)(z-1)^d}{(z+\lambda-1)^d-(z-1)^d}\biggr)^d$
where the natural number $d\geq 2$ and $\lambda \neq 0 $ is a complex parameter \cite{yang2017rational}. This family of rational maps is the renormalization transformation of the generalized diamond hierarchical Potts model.  Qiao gives it \cite{Qi} in 2014. The special case named standard diamond lattice $(d=2)$ was first studied by Derrida et al. \cite{DSI} in $1983$. Using Theorem \ref{noherman}, it can be proved that the renormalization transformation $S_\lambda$ has no Herman rings.
The Blaschke product  $F_{a,\lambda}=\lambda z^2 \frac{z-a}{1-\bar{a}z}$ of degree $3$ was investigated by Henriksen \cite{ch} in $1997$. He showed that for an irrational rotation number $\alpha$ satisfying the Brjuno condition,  there exists a constant $a\geq 3$ and a continuous function $\lambda{\alpha}$ such that $F_{a,\lambda\alpha}$ possesses a Herman ring.

All the functions without Herman rings discussed in this section have only two or three critical points, except the maps given by Lemma \ref{morecri}. Since one critical value may corresponds to  many critical point,  one should focus on critical values instead of critical points. This leads to the natural question: Can a rational map with just two or three critical values have a Herman ring in its Fatou set? By applying the Riemann-Hurwitz formula, one can show that any rational map with exactly two critical values must also have only two critical points. Consequently, Milnor’s result \cite{Milnor2000} implies that such maps cannot exhibit Herman rings. In $2019$, Hu et al. \cite{hu} explored this question further, by focusing on rational maps with three critical values. The main objective of the paper was to investigate a specific class of rational maps with three critical values, known as regularly ramified rational maps. A rational map \( f \) is said to be regularly ramified if, for every point \( z \in \widehat{\mathbb{C}} \), all pre-images of \( z \) under \( f \) have the same local degree, i.e., \( f \) has equal ramification indices at all pre-images of any point. This uniform ramification imposes strong combinatorial restrictions on the map. A classical consequence of these restrictions is that regularly ramified rational maps necessarily have at most three critical values.
On the other hand, rational maps with exactly three critical values are highly constrained and arise naturally as quotients of regularly ramified maps. In fact, such maps correspond to branched coverings of the Riemann sphere with three branch points, and by the Riemann-Hurwitz formula these coverings must have uniform ramification over each critical value. As a result, rational maps with three critical values can be realized (up to Möbius conjugacy) as compositions or quotients of regularly ramified rational maps. The following result is established in \cite{hu}.

\begin{theorem}
    No regularly ramified rational map can have a Herman ring in its Fatou set.
\end{theorem}

\subsection{Configurations of Herman rings}

Once the existence of cycles of Herman rings is established, the question of possible configurations arises. By configurations of Herman rings, we mean the location of $H$-relevant poles (poles which are enclosed by a ring), the location of Herman rings in the complex plane,  and the mapping pattern of the Herman rings. We introduce the concept of nested Herman rings.
\begin{definition}{\bf{($H$-Maximal nest)}} 
	Given a Herman ring H, a ring $H_j$ is called an $H$-outermost ring if $H_j$ is not contained in $B(H_i)$ for any $i$, $i\neq j$. Given an $H$-outermost ring $H_j$, the collection of rings consisting of $H_j$ and all $H_i$ such that $H_i \subset B(H_j)$ is called an $H$-maximal nest.
\end{definition}
In other words, an $H$-maximal nest is a sub-collection of Herman rings from the periodic cycle of $H$ such that there exists a ring $H_i$ such that $B(H_i)$ contains all the rings in the sub-collection. The number of $H$-maximal nest can be any natural number less than or equal to the period of $H$, and each $H$-maximal nest corresponds to an $H$-outermost ring. Figure \ref{figch22} demonstrates the configuration in detail whenever the function has an omitted value. The concept of an $H$-maximal nest helps us to classify the possible arrangements of a $p$-periodic cycle of Herman rings. 
\begin{definition}{\bf{(Nested, strictly nested and strictly non-nested)}}
	A $p$-periodic cycle of Herman rings with $p>1$ is called nested if there is a $j$ such that $H_i \subset B(H_j)$ for all $i \neq j$. It is called strictly nested if for each $i\neq j$, either $H_i \subset B(H_j)$ or $H_j \subset B(H_i)$. It is strictly non-nested if $B(H_i)\cap B(H_j)= \phi$ for all $i \neq j$.
\end{definition}

There is extensive work by Shishikura \cite{smm,ms,mss} concerning the configuration of Herman rings of rational functions. For a given configuration of rings, he defined an associated abstract tree. He showed that, for any tree satisfying certain conditions, one can construct a rational map with a cycle of Herman rings realizing such a tree. But, the compact form a function was not mentioned there.  In $2021$, the explicit formula of a family of rational maps with a $2$-cycle of nested Herman rings is described by Yang \cite{yang} in the following theorem.
\begin{theorem}
	For any Brjuno number $\theta$, there exists $r,s\in (0,1)$ such that $e^{2\pi i s}\frac{z-r^2e^{-2\pi is}}{1-r^2z e^{2\pi is}}({\frac{z-\frac{1}{r}}{1-\frac{z}{r}}})^3$ has a $2$-cycle of nested Herman rings of rotation number $\theta$.
\end{theorem}
In fact, for any natural number $p\geq 2$, the same article describes the example of a family of rational maps with a $p$-periodic Herman ring.

\par

\subsection{Root-finding algorithms and Herman rings}
By root-finding algorithm, we mean a rational map $T_f:Poly_d \rightarrow Rat_k$ such that the roots of the polynomial $f$ are the attracting fixed points of $T_f$ where $Poly_d$ is the set of all polynomials of degree less than or equal to $d$ and $Rat_k$ is the set of all rational maps of degree less than or equal to $k$. A classical root-finding method is Newton's method. For a given complex polynomial $f$, the rational map $T_f:\widehat{\mathbb{C}}\rightarrow \widehat{\mathbb{C}} $ defined as $$ T_f: z\mapsto z- \frac{f(z)}{f^\prime (z)} $$ is called the Newton's map for $f$. If $f \in Poly_2$ then $T_f \in Rat_2$. Since quadratic rational maps do not have any Herman rings, $T_f$ does not have any Herman rings whenever the degree of $f$ is less than or equal to $2$ \cite{smm}. In $1997$, Tan proved that even if the  degree of $f$ is $3$, then $T_f$ has no Herman rings [\cite{Tan}, Proposition 2.6]. In $2009$, Shishikura proved that the Julia set of Newton's method for every polynomial is connected. Hence, they have no Herman rings [\cite{SH}, Corollary II]. Specifically, it is proved that if the Julia set of a rational function is disconnected, then it must contain at least two weakly repelling fixed points. Thus, such rationl map can not be written as a Newton's map of a polynomial.  Note that A fixed point $z$ of $f$ is called weakly repelling if $|f'(z)| > 1 \quad \text{or} ~ f'(z) = 1$. Since Newton's map of any polynomial has at most one weakly repelling fixed point, the Julia set of Newton's method for every polynomial is connected, and hence, they have no Herman rings. A natural generalization of Newton's method is the K\"{o}nig root-finding method. For a given polynomial $f$ and a natural number $\sigma\geq 2$, the K\"{o}nig method of order $\sigma$ associated with $f$ is the rational map $K_{f,\sigma}(z)=z+(\sigma-1)\frac{(\frac{1}{f(z)})^{[\sigma-2]}}{(\frac{1}{f(z)})^{[\sigma-1]}}$. Though the existence of Herman rings of K\"{o}nig's method is not proved yet, in $2013$, Honorato constructed a map $K_{f, \sigma}$ whose Julia set is not connected  \cite{hono}. Thus, the existence of the Herman ring of $K_{f, \sigma}$ is not ruled out.
Another popular root-finding method is Chebyshev's method. For a given polynomial $f$, the Chebyshev's method $C_f$ is defined as $C_f(z)=z-(1+\frac{1}{2}L_f(z))\frac{f(z)}{f^\prime(z)}$, where $L_f(z)=\frac{f(z)f''(z)}{{f'(z)}^2}$. The existence of the Herman ring of $C_f$ is also yet  not known. To our knowledge, no root-finding method is known to possess a Herman ring in the Fatou set. The same is also true whenever $f$ is an entire function.

\section{Herman rings for transcendental functions}\label{trans}
In this section, we investigate the existence and construction of Herman rings in transcendental functions.
\subsection{ Entire functions and analytic self-maps}
It follows from Theorem \ref{nohr} that entire functions do not have any Herman ring. A  transcendental holomorphic map from $\mathbb{C^*}$ to itself is called an analytic self-map of $\mathbb{C^*}$. Note that it is of two types. If $0$ is a pole and an omitted value, then the function is of the form $\frac{e^{h(z)}}{z^n}$, where $h$ is a non-constant entire function and $n\in \mathbb{N}$. On the other hand, if  $0$ is an essential singularity, then the function is of the form $z^ne^{g(z)+h(\frac{1}{z})}$ where $g, ~h$ are non-constant entire functions and $n\in \mathbb{N}$. It is proved by Baker \cite{puncture2} in $1987$ (wandering domains in the puncture disk) that $f(z)=e^{2\pi i \beta}ze^{\alpha (z-\frac{1}{z})}$ has a Herman ring for $0<\alpha<0.5$ and for a suitable choice of $\beta$ (see Figure \ref{hr1}). Moreover, if $V$ is a multiply connected Fatou component of $f(z)=ze^{g(z)+h(\frac{1}{z})}$,  then $V$ is necessarily a Herman ring. Later, in $1996$, Gong et al. \cite{hkg} proved that if $f(z)=z^ne^{g(z)+h(\frac{1}{z})}$, then $f$ has no Herman rings whenever $n\in\mathbb{Z}\setminus\{1\}$. Consequently, the functions  $\frac{e^{h(z)}}{z^n}$ do not have any Herman rings for  a non-constant entire function $h$ and a natural number $n$.


\begin{figure}[h]
	\centering
	
	\includegraphics[width=14.5cm,height=8cm,angle=0]{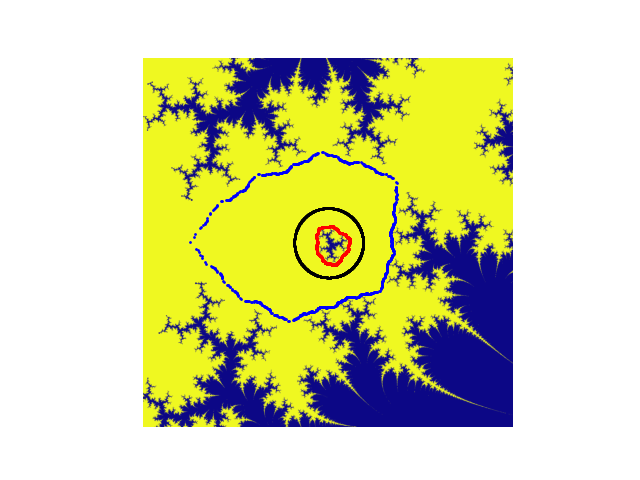}  
	
	\caption{Herman ring of $ze^{i \alpha} e^{\frac{1}{2}\beta(z-\frac{1}{z})}$ for $\alpha=1.9$ and $\beta=0.5$. The blue and red curves are the boundaries of the Herman ring (also the orbit of free critical points $\frac{-1\pm \sqrt{1-\beta^2}}{\beta}$).}
	
	\label{hr1} 
\end{figure}

\subsection{ General meromorphic functions}
Let $M$ denote the class of transcendental meromorphic functions with at least two poles or one pole that is not an omitted value. The class $M$ is usually referred to as the class of general meromorphic functions \cite{berg93}.
Transcendental meromorphic functions are known to have the Herman rings. This section discusses the existence and the number of Herman rings of a transcendental meromorphic function.  The work was stimulated by Zheng on the uniform perfectness of the Julia set of a transcendental meromorphic function of finite type. By the methods of quasiconformal deformation,  Zheng \cite{zh} in $2000$ proved the following.
\begin{theorem}
	Let $f$ be a meromorphic function of finite type. Then $f$ has only a finite number of Herman rings. 
\end{theorem}
A meromorphic function $f$  is of finite type if $f$ has finitely many singular values. It is proved that if $f$ has $k$ many singular values then $f$ can have at most $2k+4$ many Herman rings. The above theorem may not be true for a meromorphic function that is not of finite type. Using quasiconformal surgery, the same article proved the following.
\begin{theorem}
	There exists a transcendental meromorphic function that has an infinite number of Herman rings.
\end{theorem}

  We are concerned with invariant cycles of Herman rings for transcendental meromorphic functions. The  result below concerns the relation of such cycles with the number of poles of $f$. It is proved by Fagella et al. \cite{nj} in $2005$ that if $H_1, H_2,....., H_n$ be invariant Herman rings of $f$, then there exists a pole in every bounded component of $\mathbb{C}\setminus \bigcup_{i=1}^{n} H_i$. Consequently, we have the following.
%
\begin{cor}
	If a transcendental meromorphic function $f$ has $n$ poles, then $f$ has at most $n$ invariant Herman rings.
\end{cor}
This implies that the number of invariant Herman rings can not exceed the number of poles. Moreover, for any given natural number $n$, the existence of a function with $n$ invariant Herman rings is proved in \cite{pdnf}.
\begin{theorem}
	Given any $n>0$, there exists $f\in M$ with exactly $n$ poles and $n$ invariant nested Herman rings. Moreover, the rotation number of each of the rings may be any prescribed Brjuno number. 
\end{theorem}
The same article gave an example of a meromorphic map with one Herman ring and one pole, which is not omitted. In particular, it is proved that there exist values $a,b\in \mathbb{C}\setminus \{0\}$ for which the meromorphic map $az^2 \frac{e^{bz}}{z+1}$ has a Herman ring.

A Herman ring $H$ of $f$ is called unbounded on one side if an essential singularity of $f$ lies on one of the connected components  of the boundary of $H$. If both connected components contain an essential singularity of $H$, we say that $H$ is doubly unbounded. Since for $f\in M$, the function $f$ only has one essential singularity at the point at infinity, and hence, it can have on one side unbounded Herman rings. Now we give a result regarding the existence of an unbounded Herman ring. It is given by Dominguez et al. \cite{pdnf} in $2004$.

\begin{theorem}\label{260}
	There exists $f\in M$ such that $\mathcal{F}(f)$ has a Herman ring that is unbounded on one side.
\end{theorem}

\par


We now discuss the Herman rings of an elliptic function. A lattice is a collection of complex numbers that form a discrete group with respect to addition. Lattice is said to be trivial, simple, or double if the collection becomes $\{0\}$, $\mathbb{Z}$, or $\mathbb{Z} \times \mathbb{Z}$, respectively. An elliptic function is a transcendental meromorphic function that is doubly periodic with respect to a given lattice. For a comprehensive idea about elliptic functions and their properties, one can consult the classical exposition of Du Val \cite{duval}. A result regarding the non-existence of the  Herman ring of an elliptic function of order $2$ is already proved by Rocha \cite{rocha} in $2020$. Since an elliptic function is of finite type, it must have a finite number of Herman rings \cite{zh}. Then, applying quasiconformal surgery, we can obtain the following result \cite{rocha}.
\begin{theorem}
	Let $f$ be an elliptic function of order $\mathcal{O}$ greater than one. Then $f$ can have at most $(\mathcal{O}-2)$ invariant Herman rings.
\end{theorem}

\subsection{Non-invariant Herman rings}


In 2012, Fagella et al.~\cite{nj} introduced a quasiconformal surgery technique to transcendentalize rational maps possessing Herman rings. In the same work, they established the existence of nested Herman rings for transcendental meromorphic functions, including the construction of a $2$-periodic nested Herman ring. However, for periods $p \geq 2$, obtaining explicit formulas for the corresponding functions and visualizing a $p$-periodic Herman ring $H$ becomes increasingly challenging, since such a ring cannot contain any circle $C$ satisfying $f^{p}(C) = C$ \cite{yang}. As an illustration, Figure~\ref{hr3} depicts a $2$-cycle of Herman rings for the function
\[
f(z) = \frac{a-b}{b e^{b}} \frac{z^{2}}{z-a} e^{z} + b,
\]
for suitable choices of the parameters $a$ and $b$. 
\begin{figure}[H]
	\centering
	
	\includegraphics[width=14.5cm,height=6.5cm,angle=0]{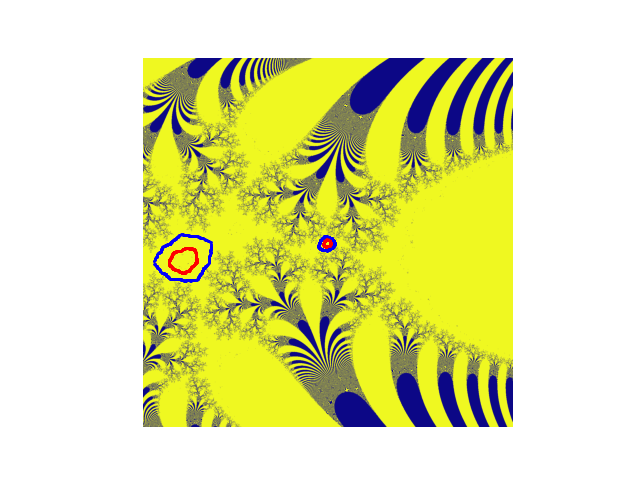}  
	
	\caption{$2-$ cycle of Herman rings of $\frac{a-b}{be^b}\frac{z^2}{z-a}e^z+b$ for $a=0.01$ and $b=-1.23797-i0.16536$. The blue and red curves denote the boundaries of the Herman rings (the orbits of $\frac{-1+a\pm \sqrt{a^2+6a+1}}{2}$).}
	
	\label{hr3} 
\end{figure}

These results motivate further investigation into explicit constructions and visualization techniques for higher-period Herman rings in the transcendental setting. The example of such functions with a non-invariant Herman ring is given by Yang \cite{yang} in the following theorem for the first time  in $2021$.

\begin{theorem}
For a positive integer $p$ and a Brjuno number $\theta$, there exist $a$, $b$, and $u=u(a,b)\in \mathbb{C}\setminus \{0\}$ such that $f_{a,b}(z)=u\frac{z^2(z-b)}{z-a}e^z$ has a $p$-cycle of Herman rings of rotation number $\theta$ and a super-attracting $p$-cycle different from $0$ (see Figure \ref{hr2}). 
\end{theorem}
\begin{figure}[H]
	\centering
	
	\includegraphics[width=14.5cm,height=6.5cm,angle=0]{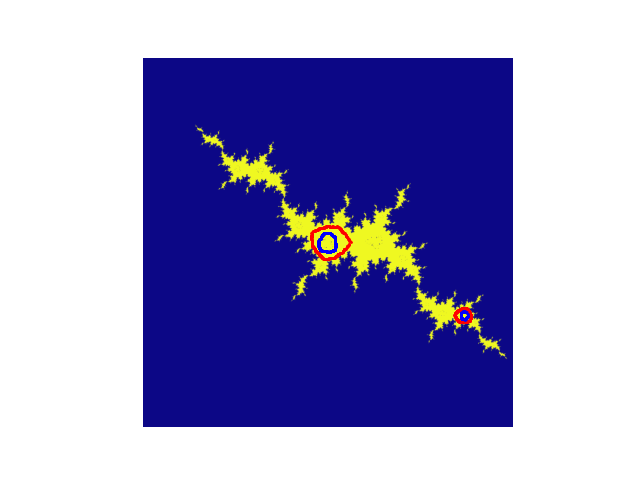}  
	
	\caption{$2-$ cycle of Herman ring of $uz^2\frac{z-a}{1-az}+b$ for $a=4$ and $b=11.03308-i 5.9193$ and $u=\frac{ab-1}{b(b-a)}$. The blue and red curves denote the boundaries of the Herman rings (the orbits of $\frac{3+a^2\pm \sqrt{a^4-10a^2+9}}{4a}$).}
	
	\label{hr2} 
\end{figure}
 
\subsection{Meromorphic functions with an omitted value}

 A point $z_{0}\in \widehat{\mathbb{C}}$ is said to be an omitted value for the function $f$ if $f(z)\neq z_{0}$ for any $z\in\mathbb{C}$. An omitted value is always an asymptotic value, but the converse need not be true. We denote the set of omitted values of $f$ as $O_f$. The dynamics of a meromorphic function is vastly affected by the presence of an omitted value, which is shown by Nayak et al. \cite{tk-zheng} in $2011$.
  Let ${M_{o}}$ be the set of functions in $M$ having at least one omitted value. Various dynamical properties of the classes ${M_{o}}$ have been studied in \cite{tk-zheng}. Nayak \cite{tk} in $2016$ investigates the existence of the Herman ring of meromorphic functions with at least one omitted value. Recall that a  Herman ring is a doubly connected periodic Fatou component. But the converse is not at all obvious. Bolsch asked the following question, which to date remains open.\\
\textbf{Question:} Is any doubly connected periodic Fatou component of a meromorphic function always a Herman ring?\\
This is known to be true when the period is one. It is given by Baker et al. \cite{inb} in $1992$. Nayak \cite{tk} in $2016$ has proved that for every $f\in M_0$, each doubly connected periodic Fatou component is a Herman ring. 

 A Jordan curve $\gamma$ is said to be non-contractible on a multiply connected Fatou component if it can not be reduced to a point under the continuous deformation of curves on the Fatou component. This is so as the bounded complementary component of $\gamma$ contains a Julia point. The following lemma, proved by Nayak et al. \cite{tk-zheng} in $2011$, is useful to track down the behavior of the Herman rings of functions with an omitted value.
\begin{lemma}\label{surroundingapole}
	 Let $V$ be a multiply connected Fatou component of $f$. Also suppose that $\gamma\subset V$ is a closed curve with $B(\gamma)\cap\mathcal{J}(f)\neq \phi.$ Then there exists an $n\in N\cup \{0\}$ and a closed curve $\gamma_{n}\subseteq f^{n}(\gamma)$ in $V_{n}$ such that $B(\gamma_{n})$ contains a pole of $f$. Further, if $O_{f}\neq \phi,$ then $O_{f}\subset B(\gamma_{n+1})$ for some closed curve $\gamma_{n+1}\
	\mbox{contained\ in}\ f(\gamma_{n}).$
\end{lemma}
\begin{figure}[H]
  \centering
  {\includegraphics[width=14.5cm,height=6cm,angle=0]{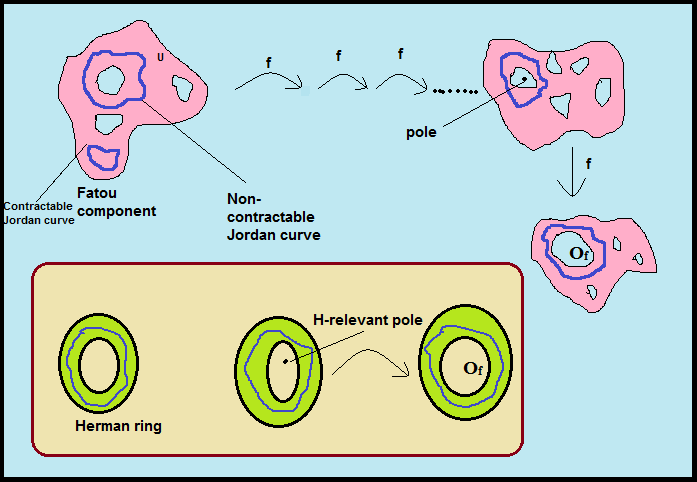}}
  \caption{Forward images of a non-contractible Jordan curve.}
  \end{figure}

It follows from the above lemma that given a Herman ring $H$, there exist rings that surround $O_f$. The following theorem gives a sufficient condition for the non-existence of the Herman ring of a function.
\begin{theorem}
	If $f\in M_o$ has only one pole, then $f$ has no Herman rings. 
\end{theorem} 
\begin{ex}
	For each entire map $g$ and non-zero complex number $z_0$, the map $\frac{e^{g(z)}}{(z-z_0)^k}$ is meromorphic with only a single pole $z_0$ which is different from its omitted value, where $k$ is a natural number. By the above theorem, it has no Herman rings.
\end{ex}
Every cycle of Herman rings of $f\in M_o$ contains a ring that encloses a pole (called an $H$-relevant pole), and the forward image of this ring encloses $O_f$. It is proved that if $H$ is a Herman ring of $f\in M_o$, then $f: B(H) \rightarrow \widehat{\mathbb{C}}$ is one-to-one \cite{tk}. Thus, the number of $H$-relevant poles is less than or equal to that of $H$-maximal nests, and each maximal nest contains at most one pole. As a consequence of the above, we have that if all the poles of a function belonging to $M_o$ are multiple, then it has no Herman rings. Using the Herman ring's basic properties, the Herman ring's non-existence of some classes of meromorphic maps is also shown in \cite{small}.  
\begin{ex}
	\begin{enumerate}
		\item 
		For any polynomial $P(z)$ with real coefficients having no real root, the function $f(z)=\frac{e^z}{P(z)}$ has no Herman rings of any period.
		\item For any polynomial $P(z)$ having all real roots, the function $f(z)=\frac{e^z}{P(z)}$ has no Herman rings of any period.
	\end{enumerate}
\end{ex}

\subsection{Configurations of Herman rings}

In this section, we describe the configurations of Herman rings of a function with an omitted value. Note that a ring $H_i$   in an $H$-maximal nest N is called innermost if it does not surround $H_j$ for any $j\neq i$.  It is clear from the definition of maximal nest that a Herman ring $H$ is nested if there is only one $H$-maximal nest. A possible configuration of Herman rings of a function with an omitted value is shown in Figure \ref{figch22}.

\begin{figure}[H]
	\centering
	{\includegraphics[width=14cm,height=7cm,angle=0]{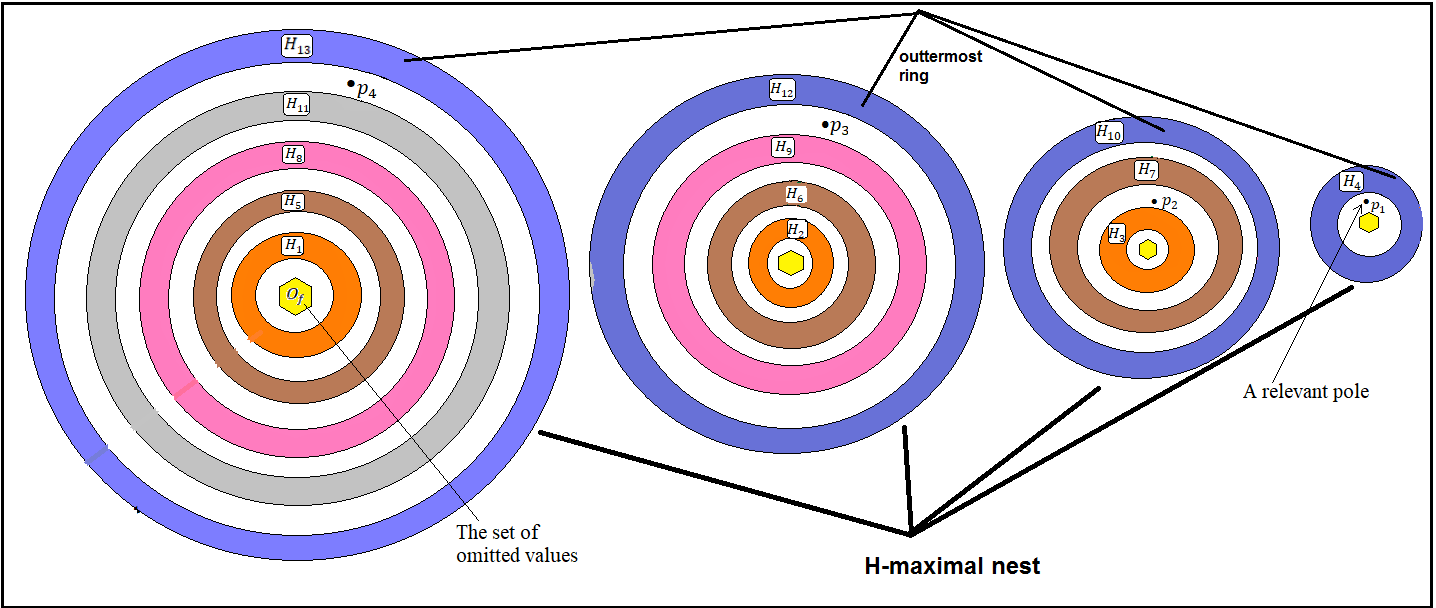}}
	\caption{$H$-maximal nest. }\label{figch22}
\end{figure}

Note that each Herman ring of period 2 is either nested or strictly non-nested.  The following result discussed about the non-existence of nested or strictly non-nested Herman rings.
\begin{theorem}
	If $f\in M_0$, then $f$ has no Herman rings that is nested or strictly non-nested, and in particular, it has no Herman rings of period one and two. Further, if a pole of $f$ is an omitted value, it has no Herman rings of any period.
\end{theorem}\label{nested&non-nested}
The proof contains a detailed analysis of the possible arrangements of Herman rings relative to each other in the plane. The locations of the omitted value(s) and poles surrounded by Herman rings have also been the key to several useful observations. For a Herman ring $H$ of a function belonging to $M_0$, if there are only two $H$-maximal nests, one of which consists of only one ring and the other is strictly nested, then $H$ is odd periodic \cite{small}.   Later, these observations are used to show that if $f\in M_o$ and $p=3$ or $4$, then the number of $p$-cycles of Herman rings is at most one. Furthermore, the simultaneous existence of a $3$-periodic and a particular type of $4$-periodic Herman rings is also ruled out for any function in $M_o$.  Possible configurations of a cycle of Herman rings of period $3$ are described in the following lemma (Lemma $3.1$, \cite{small}).
\begin{lemma}\label{3-periodic}
	Let $H$ be a $3$-periodic Herman ring of $f\in M_o$ and $\{H=H_0,H_1,H_2\}$ be the cycle containing $H$. Then there are exactly two $H$-relevant poles, say $p_1,p_2$, and two $H$-maximal nests, say $N_1$, $N_2$, containing one and two rings respectively. If $H_1\in N_1 $, then $H_2$ is the outermost ring of $N_2$ and $H$ is the innermost ring of $N_2$. Further, $p_1\in B(H_1)$, $p_2\in B(H_2)\setminus B(H_0)$ and $O_f\subset B(H_0)$ (see Figure \ref{fig:3periodic}).
\end{lemma}
   
\begin{figure}[H]
\centering
\includegraphics[width=14.5cm,height=6cm,angle=0]{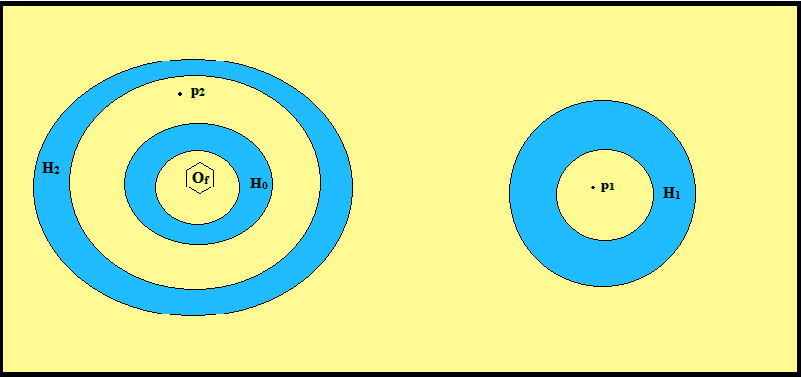}
\caption{Cycle of $3$-periodic Herman rings.}
\label{fig:3periodic}
\end{figure}
In \cite{gc}, authors shown that the possible number of configurations of a $5$-periodic Herman ring of a function in $M_o$ is six and there can not have two different cycles of $5$-periodic Herman rings simultaneously. Regarding possible configurations of Herman rings, we have the following question:\\
\textbf{Question:} Is there any bound on the possible number of configurations of an $n$-cycle of Herman rings of period $n\in\mathbb{N}$?

Chakraborty et al. \cite{bovg} in $2022$ proved that the existence of a  quasi-nested wandering domain of a meromorphic function $f$ with an omitted value ensures that there exists no unbounded Fatou component of $f$ giving that an unbounded Herman ring does not exist for $f$. We make the following conjecture.\\
\textbf{Conjecture:} Meromorphic functions with quasi-nested wandering domains have no Herman rings.
\par

Using the results of \cite{small,tk}, Ghora and Nayak have proved that a particular subset of rings (defined as a basic chain) of a cycle of Herman rings mostly controls the behavior of the cycle in \cite{relevantpole}. Based on that, they found a lower bound on the period of a Herman ring 
\begin{theorem}
    Let $H$ be a $p$-periodic Herman ring  of a function $f$, and let $h$ be the number of $H$-relevant poles. Then $p \geq \frac{h(h+1)}{2}$. In particular, the following are true.
\begin{enumerate}
\item If the basic nest (the nest surrounding the $O_f$) surrounds a pole, then $p\geq \frac{h(h+1)}{2}$.
\item If the basic nest does not surround any pole, then $p\geq \frac{h (h+3)}{2}$.
\end{enumerate}
\label{lowerbound}
\end{theorem}
They also found that the period of a Fatou component whose closure contained an omitted value depends on the length of the basic chain of a cycle of rings.
\begin{theorem}
Let  $U$ be a periodic Fatou component of $f$ such that its closure contains at least one omitted value. Then for every Herman ring $H$ of $f$, the number of $H$-relevant poles is strictly less than the period of $U$. In particular, if $U$ is invariant or $2$-periodic, then $f$ has no Herman rings.
\label{periodicfatoucomponent}
\end{theorem}
On the presence of a particular type of omitted value, namely the Baker omitted value (bov), Ghora et al. \cite{ghora2} in 2021  explore the existence of Herman rings. A Baker omitted value  $b \in  \mathbb{C}$ of a meromorphic function $f$ is an omitted value of $f$ such that there is a disk $D$ with center at $b$ for which each component of the boundary of $f^{-1}(D)$ is bounded. The bov is called stable if it is contained in the Fatou set of the function. They proved the following theorem.

 \begin{theorem}\label{pperiodicHR}
	Let  $f$ be a function with a stable bov and $p \geq 3$. 
	Then the number of $p$-periodic Herman rings is finite.
\end{theorem}

Further, it is shown that if the function has a stable bov and the bov is the only limit point of the critical values, then the number of Herman rings whose boundary intersects the forward orbit of a critical value is finite.

It is evident through out this section that many sufficient conditions are found to ensure the non-existence of Herman rings of a function in $M_o$. Based on that, our intuition is that if $f\in M_o$ then it cannot have a Herman ring. We propose the following conjecture.\\
\textbf{Conjecture:} If a general meromorphic function has a finite omitted value, it does not have a Herman ring of any period.

\section{Conclusion}\label{con}
 Herman rings provide essential insights into the detailed structure of the Fatou set, demonstrating the various sorts of rich behavior in holomorphic mappings. Tools such as quasiconformal surgery, rotation number analysis, and invariant curves have proven invaluable in constructing and analyzing Herman rings. However, many topics remain unanswered, such as the complete categorization of rational functions that admit Herman rings, and their parameter space behavior remains a fertile ground for investigation. We conclude by making a few remarks that lead to some research directions.
 
 \begin{enumerate}
 	\item Siegel discs and Herman rings show quasi-periodic dynamics, but the latter are seen in annular domains. This adds topological complexity and richness to dynamical classifications. Furthermore, a specific kind of periodic point known as a Siegel point guarantees the existence of a Siegel disk. Is there a point whose existence will guarantee the Herman ring's existence?
 	
 	\item It is known that rational functions of degree $2$ cannot have any Herman rings. However, further investigation is needed to classify all rational functions that can admit Herman rings, particularly in low degrees (degree $3$ or $4$), where examples are still limited. Despite several known constructions, a complete characterization of rational maps admitting Herman rings is still lacking. In particular, it remains open to identify necessary and sufficient conditions in terms of degree, critical orbit configurations, and combinatorial data. It is also fascinating to investigate parameter slices for families of functions with a Herman ring. Using Teichm\"{u}ller theory and quasiconformal deformation theory to map out these regions can also be explored.
	
	\item The behavior of Herman rings under perturbations of the underlying map is not well understood. In particular, the mechanisms by which Herman rings are created or destroyed in parameter space remain to be clarified. Also, the structure and size of the set of parameters corresponding to maps with Herman rings in natural families of rational and transcendental maps are still largely unknown.
	
	\item The geometric nature of the boundary components of Herman rings is largely mysterious. In particular, it is not known under what conditions these boundaries are Jordan curves, quasicircles, or have fractal structure. Little is known about the Hausdorff dimension or Lebesgue measure of the boundary of a Herman ring. Whether such boundaries can have positive area is an open question.
	
	\item A conformal mating is a construction that “glues” the Julia set of a function to the action of a Fuchsian group (or subgroup of $PSL(2,\mathbb{Z})$) on the upper half-plane, creating a single function that incorporates both behaviors consistently-conformally and dynamically. Conformal mating in the transcendental setting is far less developed than in the polynomial or rational case. Hence, one can study to develop an obstruction theory and identify structural conditions under which such matings could theoretically exist. The goal should be to determine when a topological mating between the  map with a Herman ring and the group action can be defined, and whether such a topological mating can be promoted to a conformal one. What distinguishes the correspondence map whenever the function has a Herman ring? Even in the absence of realizable matings, identifying dynamical obstructions--such as incompatibility of laminations, singular value distributions, or boundary correspondences---will yield meaningful insights into the limitations of transcendental matings. This exploratory direction, though ambitious, is grounded in concrete structural questions and offers potential for developing a new interface between transcendental dynamics and hyperbolic geometry.

	
	\item The presence of a Herman ring strongly suggests, but does not automatically guarantee, the existence of buried Julia components. In many known constructions, rational maps with Herman rings do exhibit buried Julia components arising from this topological separation. However, the existence of buried Julia components depends on the global arrangement of Fatou components and how the Julia set accumulates on their boundaries. It remains a subtle problem to determine whether every rational map with a Herman ring must contain buried Julia components or whether additional dynamical or topological conditions are required. As such, the question of whether Herman rings ensure the presence of buried Julia components is an important and largely open direction of investigation.
	
	\item Certain root-finding techniques are known to have disconnected Julia sets. This raises the question of whether there is a root-finding method whose Fatou set contains a Herman ring.
	
	\item Herman ring is a doubly connected periodic Fatou component. However, the opposite has not been proven yet.  Is every doubly connected periodic Fatou component of a meromorphic function always a Herman ring?

	\item  For a given Herman ring $H$, Theorem \ref{lowerbound} provides a lower bound on the period of $H$ in terms of the number of $H$-relevant poles. Is it possible to find an upper bound on the period of $H$? This question seems to be more relevant for meromorphic maps with finitely many poles and with an omitted value. Such a function is of the form  $f(z)=\frac{E(z)}{(z-z_0)(z-z_1)...(z-z_p)}$, $p>1$ and $z_0, z_1,\dots, z_p \in \mathbb{C}$ where $E$ is an entire function omitting a finite point.

	 \item For a function with a stable bov and fixed $p\geq 3$, Theorem \ref{pperiodicHR} proves that the number of $p$-periodic Herman rings is finite. However, this result does not exclude the possibility of the function having infinitely many Herman rings of different periods. Whether such a situation can occur, particularly when the bov is the only accumulation point of the critical values, remains an open question.

	\item Theorem \ref{periodicfatoucomponent} provides a necessary condition under which Herman rings cannot exist. In addition to this, several other non-existence results are known (see, for instance, \cite{tk, tk-zheng}). However, no known example exists of a function that omits a value and possesses a Herman ring. This has led to the conjecture that a meromorphic function with an omitted value cannot have a Herman ring.
	\par 	Most known examples of transcendental meromorphic functions with Herman rings are obtained through quasiconformal surgery. However, it is generally challenging to determine whether such constructed functions omit any value. So far, no example constructed in this way has been found to have an omitted value. Another compelling reason supporting the conjecture lies in a shared property among functions known not to have Herman rings. It is well-known that there is no Herman ring in the Fatou set of any polynomial, analytic self-map of the punctured plane, or transcendental entire function.  Each of these functions has a point with a finite backward orbit, namely $\infty$. In the case of a function with an omitted value, the backward orbit of the omitted value is empty. This standard feature, a point with a finite backward orbit, appears to obstruct the existence of Herman rings. For this reason, the conjecture can be considered heuristically valid, even if a complete proof is still lacking.
	
\end{enumerate}
\newpage
\bibliographystyle{plainurl}
\bibliography{ref1}

\begin{thebibliography}{10}

\bibitem{ahlfors2006lectures}
L.~V. Ahlfors.
\newblock {\em {Lectures on quasiconformal mappings}}, volume~38 of {\em
  University Lecture Series}.
\newblock American Mathematical Society, 2006.
\newblock \href {http://dx.doi.org/10.1090/ulect/038}
  {\path{doi:10.1090/ulect/038}}.

\bibitem{puncture2}
I.~N. Baker.
\newblock {Wandering domains for maps of the punctured plane}.
\newblock {\em Annales Fennici Mathematici}, 12(2):191--198, 1987.
\newblock \href {http://dx.doi.org/10.5186/aasfm.1987.1204}
  {\path{doi:10.5186/aasfm.1987.1204}}.

\bibitem{inb}
I.~N. Baker, J.~Kotus, and L.~Yinian.
\newblock {Iterates of meromorphic functions IV: Critically finite functions}.
\newblock {\em Results in Mathematics}, 22(3-4):651--656, 1992.
\newblock \href {http://dx.doi.org/10.1007/BF03323112}
  {\path{doi:10.1007/BF03323112}}.

\bibitem{bb}
R.~Bam{\'o}n and J.~Bobenrieth.
\newblock {The rational maps $z\to 1+\frac{1}{\omega z^d}$ have no Herman
  rings}.
\newblock {\em Proceedings of the American Mathematical Society},
  127(2):633--636, 1999.
\newblock \href {http://dx.doi.org/10.1090/S0002-9939-99-04566-9}
  {\path{doi:10.1090/S0002-9939-99-04566-9}}.

\bibitem{beardon2000iteration}
A.~F. Beardon.
\newblock {\em {Iteration of rational functions: Complex analytic dynamical
  systems}}, volume 132 of {\em Graduate Texts in Mathematics}.
\newblock Springer Science \& Business Media, 2000.

\bibitem{berg93}
W.~Bergweiler.
\newblock {Iteration of meromorphic functions}.
\newblock {\em Bulletin of the American Mathematical Society}, 29(2):151--188,
  1993.
\newblock \href {http://dx.doi.org/10.1090/S0273-0979-1993-00432-4}
  {\path{doi:10.1090/S0273-0979-1993-00432-4}}.

\bibitem{boettcher1904principal}
L.~E. Boettcher.
\newblock {The principal laws of convergence of iterates and their application
  to analysis}.
\newblock {\em Izv. Kazan. Fiz.-Mat. Obshch}, 14:155--234, 1904.

\bibitem{branner2014quasiconformal}
B.~Branner and N.~Fagella.
\newblock {\em {Quasiconformal surgery in holomorphic dynamics}}, volume 141 of
  {\em Cambridge Studies in Advanced Mathematics}.
\newblock Cambridge University Press, 2014.
\newblock \href {http://dx.doi.org/10.1017/CBO9781107337602}
  {\path{doi:10.1017/CBO9781107337602}}.

\bibitem{cayley1879application}
A.~Cayley.
\newblock {Application of the Newton-Fourier method to an imaginary root of an
  equation}.
\newblock {\em Quart. J. Pure Appl. Math}, 16:179--185, 1879.

\bibitem{small}
T.~K. Chakra, G.~Chakraborty, and T.~Nayak.
\newblock {Herman rings with small periods and omitted values}.
\newblock {\em Acta Mathematica Scientia}, 38(6):1951--1965, 2018.
\newblock \href {http://dx.doi.org/10.1016/S0252-9602(18)30858-0}
  {\path{doi:10.1016/S0252-9602(18)30858-0}}.

\bibitem{bovg}
G.~Chakraborty and S.~K. Datta.
\newblock {On quasi-nested wandering domains}.
\newblock {\em Filomat}, 36(8):2687--2694, 2022.
\newblock \href {http://dx.doi.org/10.2298/FIL2208687C}
  {\path{doi:10.2298/FIL2208687C}}.

\bibitem{gc}
G~Chakraborty, S.~K. Datta, and S.~Sahoo.
\newblock {Configurations of Herman rings in the complex plane}.
\newblock {\em Indian J. Math}, 63(3):375--391, 2021.

\bibitem{cremer1928zentrumproblem}
H.~Cremer.
\newblock {Zum zentrumproblem}.
\newblock {\em Mathematische Annalen}, 98(1):151--163, 1928.

\bibitem{DSI}
B.~Derrida, L.~De~Seze, and C.~Itzykson.
\newblock {Fractal structure of zeros in hierarchical models}.
\newblock {\em Journal of Statistical Physics}, 33:559--569, 1983.
\newblock \href {http://dx.doi.org/10.1007/BF01018834}
  {\path{doi:10.1007/BF01018834}}.

\bibitem{pdnf}
P.~Dom{\'\i}nguez and N.~Fagella.
\newblock {Existence of Herman rings for meromorphic functions}.
\newblock {\em Complex Variables, Theory and Application: An International
  Journal}, 49(12):851--870, 2004.
\newblock \href {http://dx.doi.org/10.1080/02781070412331298589}
  {\path{doi:10.1080/02781070412331298589}}.

\bibitem{duval}
P.~Du~Val.
\newblock {\em {Elliptic functions and elliptic curves}}, volume~9 of {\em
  London Mathematical Society Lecture Note Series}.
\newblock Cambridge University Press, 1973.
\newblock \href {http://dx.doi.org/10.1017/CBO9781107359901}
  {\path{doi:10.1017/CBO9781107359901}}.

\bibitem{nj}
N.~Fagella and J.~Peter.
\newblock {On the configuration of Herman rings of meromorphic functions}.
\newblock {\em Journal of Mathematical Analysis and Applications},
  394(2):458--467, 2012.
\newblock \href {http://dx.doi.org/10.1016/j.jmaa.2012.05.005}
  {\path{doi:10.1016/j.jmaa.2012.05.005}}.

\bibitem{fatou1919equations}
P.~Fatou.
\newblock {Sur les {\'e}quations fonctionnelles Premier m{\'e}moire}.
\newblock {\em Bulletin de la Soci{\'e}t{\'e} math{\'e}matique de France},
  47:161--271, 1919.
\newblock \href {http://dx.doi.org/10.24033/bsmf.1003}
  {\path{doi:10.24033/bsmf.1003}}.

\bibitem{fatou1926iteration}
P.~Fatou.
\newblock {Sur l'it{\'e}ration des fonctions transcendantes enti{\`e}res}.
\newblock 1926.

\bibitem{ghora2}
S.~Ghora and T.~Nayak.
\newblock {Rotation domains and Stable Baker omitted value}.
\newblock {\em Qualitative Theory of Dynamical Systems}, 20(3):92, 2021.
\newblock \href {http://dx.doi.org/10.1007/s12346-021-00527-0}
  {\path{doi:10.1007/s12346-021-00527-0}}.

\bibitem{relevantpole}
S.~Ghora and T.~Nayak.
\newblock {On periods of Herman rings and relevant poles}.
\newblock {\em Indian Journal of Pure and Applied Mathematics}, 53(2):505--513,
  2022.
\newblock \href {http://dx.doi.org/10.1007/s13226-021-00112-w}
  {\path{doi:10.1007/s13226-021-00112-w}}.

\bibitem{ghora}
S.~Ghora and T.~Nayak.
\newblock {Iteration of Some Topologically Hyperbolic Meromorphic Maps with
  Infinitely Many Singular Values}.
\newblock {\em Complex Analysis and Operator Theory}, 19, 2024.
\newblock \href {http://dx.doi.org/10.1007/s11785-024-01647-6}
  {\path{doi:10.1007/s11785-024-01647-6}}.

\bibitem{hkg}
Z.~Gong, W.~Qiu, and F.~Ren.
\newblock {A negative answer to a problem of Bergweiler}.
\newblock {\em Complex Variables, Theory and Application: An International
  Journal}, 30(4):315--322, 1996.
\newblock \href {http://dx.doi.org/10.1080/17476939608814933}
  {\path{doi:10.1080/17476939608814933}}.

\bibitem{ch}
C.~Henriksen.
\newblock {\em {Holomorphic Dynamics and Herman Rings}}.
\newblock 1997.

\bibitem{herman1979conjugaison}
M.~R. Herman.
\newblock {Sur la conjugaison diff{\'e}rentiable des diff{\'e}omorphismes du
  cercle {\`a} des rotations}.
\newblock {\em Publications Math{\'e}matiques de l'IH{\'E}S}, 49:5--233, 1979.
\newblock URL: \url{http://www.numdam.org/item/PMIHES_1979__49__5_0/}.

\bibitem{herman1984exemples}
M.~R. Herman.
\newblock {Exemples de fractions rationnelles ayant une orbite dense sur la
  sph{\`e}re de Riemann}.
\newblock {\em Bulletin de la Soci{\'e}t{\'e} math{\'e}matique de France},
  112:93--142, 1984.

\bibitem{hono}
G.~Honorato.
\newblock {On the Julia set of K\"{o}nig's root-finding algorithms}.
\newblock {\em Proc. Amer. Math. Soc.}, 141:3601--3607, 2013.
\newblock URL: \url{http://eudml.org/doc/234994}.

\bibitem{hu}
J.~Hu and Y.~Xiao.
\newblock No herman rings for regularly ramified rational maps.
\newblock {\em Proc. Amer. Math. Soc.}, 147(1):1587--1596, 2019.
\newblock \href {http://dx.doi.org/10.1090/proc/14347}
  {\path{doi:10.1090/proc/14347}}.

\bibitem{julia1918memoire}
G.~Julia.
\newblock {M{\'e}moire sur l'it{\'e}ration des fonctions rationnelles}.
\newblock {\em Journal de math{\'e}matiques pures et appliqu{\'e}es},
  1:47--245, 1918.
\newblock URL: \url{http://eudml.org/doc/234994}.

\bibitem{khin}
A.~Ya. Khintchine.
\newblock {\em {Continued fractions}}, volume~1 of {\em Translated by Peter
  Wynn. P. Noordhoff Ltd.}
\newblock Groningen, 1963.
\newblock \href {http://dx.doi.org/10.1126/science.145.3631.478}
  {\path{doi:10.1126/science.145.3631.478}}.

\bibitem{koenigs1884recherches}
G.~Koenigs.
\newblock {Recherches sur les int{\'e}grales de certaines {\'e}quations
  fonctionnelles}.
\newblock In {\em Annales scientifiques de l'Ecole normale sup{\'e}rieure},
  volume~1, pages 1--41, 1884.

\bibitem{leau1897etude}
L.~Leau.
\newblock {{\'E}tude sur les {\'e}quations fonctionnelles {\`a} une ou {\`a}
  plusieurs variables}.
\newblock {\em Annales de la Facult{\'e} des sciences de Toulouse pour les
  sciences math{\'e}matiques et les sciences physiques}, 11(3):E25--E110, 1897.

\bibitem{lm}
M.~Lyubich.
\newblock {The dynamics of rational transforms: the topological picture}.
\newblock {\em Russian Mathematical Surveys}, 41(4):43, 1986.
\newblock \href {http://dx.doi.org/10.1070/RM1986v041n04ABEH003376}
  {\path{doi:10.1070/RM1986v041n04ABEH003376}}.

\bibitem{mane}
R.~Mane.
\newblock {On a Theorem of Fatou}.
\newblock {\em Bol. Soc. Bras. Mat.}, 24:1--11, 1992.
\newblock \href {http://dx.doi.org/10.1007/BF01231694}
  {\path{doi:10.1007/BF01231694}}.

\bibitem{may}
J.P. May.
\newblock {\em {A Concise Course in Algebraic Topology}}, volume~1.
\newblock The University of Chicago Press, 1999.
\newblock \href {http://dx.doi.org/10.1090/S0273-0979-00-00899-5}
  {\path{doi:10.1090/S0273-0979-00-00899-5}}.

\bibitem{Milnor2000}
J.~Milnor.
\newblock On rational maps with two critical points.
\newblock {\em Experimental Mathematics}, 9(4):481--522, 2000.
\newblock \href {http://dx.doi.org/10.1080/10586458.2000.10504657}
  {\path{doi:10.1080/10586458.2000.10504657}}.

\bibitem{jm}
J.~Milnor.
\newblock {\em {Dynamics in one complex variable}}, volume 160 of {\em Annals
  of Mathematics Studies}.
\newblock Princeton University Press, 2011.

\bibitem{tk}
T.~Nayak.
\newblock {Herman rings of meromorphic maps with an omitted value}.
\newblock {\em Proceedings of the American Mathematical Society},
  144(2):587--597, 2016.
\newblock \href {http://dx.doi.org/10.1090/proc12715}
  {\path{doi:10.1090/proc12715}}.

\bibitem{tk-zheng}
T.~Nayak and J.~H. Zheng.
\newblock {Omitted values and dynamics of meromorphic functions}.
\newblock {\em Journal of the London Mathematical Society}, 83(1):121--136,
  2011.
\newblock \href {http://dx.doi.org/10.1112/jlms/jdq065}
  {\path{doi:10.1112/jlms/jdq065}}.

\bibitem{Qi}
J.~Qiao.
\newblock {Julia sets and complex singularities of free energies}.
\newblock {\em Memoirs of the American Mathematical Society}, 234(1102), 2015.
\newblock \href {http://dx.doi.org/10.1090/memo/1102}
  {\path{doi:10.1090/memo/1102}}.

\bibitem{rocha}
M.~M. Rocha.
\newblock {Herman rings of elliptic functions}.
\newblock {\em Arnold Mathematical Journal}, 6(3):551--570, 2020.
\newblock \href {http://dx.doi.org/10.1007/s40598-020-00167-3}
  {\path{doi:10.1007/s40598-020-00167-3}}.

\bibitem{smm}
M.~Shishikura.
\newblock {On the quasiconformal surgery of rational functions}.
\newblock {\em Annales scientifiques de l'{\'E}cole Normale Sup{\'e}rieure},
  20(1):1--29, 1987.
\newblock \href {http://dx.doi.org/10.24033/asens.1522}
  {\path{doi:10.24033/asens.1522}}.

\bibitem{ms}
M.~Shishikura.
\newblock {Trees associated with the configuration of Herman rings}.
\newblock {\em Ergodic Theory and Dynamical Systems}, 9(3):543--560, 1989.
\newblock \href {http://dx.doi.org/10.1017/S0143385700005174}
  {\path{doi:10.1017/S0143385700005174}}.

\bibitem{mss}
M.~Shishikura.
\newblock {A new tree associated with Herman rings}.
\newblock In {\em Complex dynamics and related fields}, pages 74--92. Research
  Institute for Mathematical Sciences, Kyoto University, 2002.

\bibitem{SH}
M.~Shishikura.
\newblock {The Connectivity of the Julia Set and Fixed Points}.
\newblock In {\em Complex Dynamics, Families and Friends}, pages 257--276. A K
  Peters/CRC Press, 2009.
\newblock \href {http://dx.doi.org/10.1201/b10617} {\path{doi:10.1201/b10617}}.

\bibitem{siegel1942iteration}
C.~L. Siegel.
\newblock {Iteration of analytic functions}.
\newblock {\em Annals of Mathematics}, 43(4):607--612, 1942.

\bibitem{Tan}
L.~Tan.
\newblock {Branched coverings and cubic Newton maps}.
\newblock {\em Fundamenta mathematicae}, 154(3):207--260, 1997.
\newblock \href {http://dx.doi.org/10.4064/fm-154-3-207-260}
  {\path{doi:10.4064/fm-154-3-207-260}}.

\bibitem{ywq}
Y.~Xiao and W.~Qiu.
\newblock {The rational maps $F_\lambda (z)= z^m+ \lambda/z^d$ have no Herman
  rings}.
\newblock {\em Proceedings-Mathematical Sciences}, 120(4):403--407, 2010.
\newblock \href {http://dx.doi.org/10.1007/s12044-010-0044-x}
  {\path{doi:10.1007/s12044-010-0044-x}}.

\bibitem{yang2017rational}
F.~Yang.
\newblock {Rational maps without Herman rings}.
\newblock {\em Proceedings of the American Mathematical Society},
  145(4):1649--1659, 2017.
\newblock \href {http://dx.doi.org/10.1090/proc/13336}
  {\path{doi:10.1090/proc/13336}}.

\bibitem{yang}
F.~Yang.
\newblock {On the formulas of meromorphic functions with periodic Herman
  rings}.
\newblock {\em Mathematische Annalen}, 384(3):989--1015, 2022.
\newblock \href {http://dx.doi.org/10.1007/s00208-021-02308-1}
  {\path{doi:10.1007/s00208-021-02308-1}}.

\bibitem{yoc}
J.~C. Yocooz.
\newblock {Analytic linearization of circle diffeomorphisms}.
\newblock 125--173, 1998.
\newblock \href {http://dx.doi.org/10.1007/978-3-540-47928-4_3}
  {\path{doi:10.1007/978-3-540-47928-4_3}}.

\bibitem{zh}
J.~H Zheng.
\newblock {Remarks on Herman rings of transcendental meromorphic functions}.
\newblock {\em Indian Journal of Pure and Applied Mathematics}, 31(7):747--752,
  2000.

\end{thebibliography}
\end{document}